\documentclass[12pt,a4paper,reqno]{amsart}

\usepackage[margin = 1in]{geometry}
\usepackage{a4wide}
\usepackage{amssymb,amsmath,amsthm,mathrsfs}
\usepackage{graphicx}
\usepackage{caption}
\usepackage{subcaption}
\usepackage{float}
\usepackage{colortbl}
\usepackage{enumerate}
\usepackage{upref}
\usepackage[bookmarks=false]{hyperref}
\usepackage[T2A,T1]{fontenc}
\DeclareSymbolFont{cyrillic}{T2A}{cmr}{m}{n}
\DeclareMathSymbol{\D}{\mathalpha}{cyrillic}{196}

\theoremstyle{plain}
\newtheorem{theorem}{Theorem}[section]
\newtheorem{lemma}[theorem]{Lemma}

\newtheorem{proposition}[theorem]{Proposition}

\theoremstyle{definition}
\newtheorem{definition}{Definition}

\newtheorem*{condition}{Condition}

\theoremstyle{remark}
\newtheorem{remark}[theorem]{Remark}

\usepackage{enumitem}
\makeatletter
\def\namedlabel#1#2{\begingroup
   #2%
 \def\@currentlabel{#2}%
   \phantomsection\label{#1}\endgroup
}

\newcommand{\floor}[1]{\lfloor#1\rfloor}
\newcommand{\set}[1]{\left\{#1\right\}}
\newcommand{\abs}[1]{\left|#1\right|}

\def\R{\ensuremath{\mathbb R}}
\def\N{\ensuremath{\mathbb N}}

\def\I{\ensuremath{{\bf 1}}}
\def\e{{\ensuremath{\rm e}}}

\def\S{\ensuremath{\mathcal S}}
\def\RR{\ensuremath{\mathcal R}}

\def\C{\ensuremath{\mathcal C}}

\def\p{\ensuremath{\mathbb P}}

\def\A{\ensuremath{A^{(q)}}}

\def\X{\mathcal{X}}

\def\nn{\ensuremath{\mathscr N}}

\def\ie{{\em i.e.}, }

\def\dist{\ensuremath{\text{dist}}}

\def\eps{\varepsilon}

\numberwithin{equation}{section}

\begin{document}

\title{Extremal dichotomy for uniformly hyperbolic systems}

\author[M. Carvalho]{Maria Carvalho}
\address{Maria Carvalho\\ Centro de Matem\'{a}tica \& Faculdade de Ci\^encias da Universidade do Porto\\ Rua do
Campo Alegre 687\\ 4169-007 Porto\\ Portugal}
\email{mpcarval@fc.up.pt}

\author[A. C. M. Freitas]{Ana Cristina Moreira Freitas}
\address{Ana Cristina Moreira Freitas\\ Centro de Matem\'{a}tica \&
Faculdade de Economia da Universidade do Porto\\ Rua Dr. Roberto Frias \\
4200-464 Porto\\ Portugal} \email{\href{mailto:amoreira@fep.up.pt}{amoreira@fep.up.pt}}
\urladdr{\url{http://www.fep.up.pt/docentes/amoreira/}}

\author[J. M. Freitas]{Jorge Milhazes Freitas}
\address{Jorge Milhazes Freitas\\ Centro de Matem\'{a}tica \& Faculdade de Ci\^encias da Universidade do Porto\\ Rua do
Campo Alegre 687\\ 4169-007 Porto\\ Portugal}
\email{\href{mailto:jmfreita@fc.up.pt}{jmfreita@fc.up.pt}}
\urladdr{\url{http://www.fc.up.pt/pessoas/jmfreita/}}

\author[M. Holland]{Mark Holland}
\address{Mark Holland\\ Department of Mathematics (CEMPS)\\
Harrison Building (327)\\
North Park Road\\
Exeter, EX4 4QF\\ UK} \email{M.P.Holland@exeter.ac.uk}
\urladdr{http://empslocal.ex.ac.uk/people/staff/mph204/}

\author[M. Nicol]{Matthew Nicol}
\address{Matthew Nicol\\ Department of Mathematics\\
University of Houston\\
Houston\\
TX 77204\\
USA} \email{nicol@math.uh.edu}
\urladdr{http://www.math.uh.edu/~nicol/}

\thanks{ACMF was partially supported by FCT grant SFRH/BPD/66174/2009. JMF was partially supported by FCT grant SFRH/BPD/66040/2009. Both these grants are financially supported by the program POPH/FSE. ACMF and JMF are supported by FCT (Portugal) project PTDC/MAT/120346/2010, which is financed by national and European structural funds through the programs  FEDER and COMPETE. MC, ACMF and JMF are also supported by CMUP, which is financed by FCT (Portugal) through the programs POCTI and POSI, with national
and European structural funds, under the project PEst-C/MAT/UI0144/2013. JMF would like to thank Mike Todd for helpful comments and suggestions.}

\date{\today}

\keywords{Extreme Value Theory, Return Time Statistics, Stationary Stochastic Processes, Metastability} \subjclass[2000]{37A50, 60G70, 37B20, 60G10, 37C25.}


\begin{abstract}
We consider the extreme value theory of a hyperbolic toral automorphism $T:  \mathbb{T}^2 \to \mathbb{T}^2$ showing that if a H\"older observation $\phi$ which is a function of a Euclidean-type
distance to a non-periodic point $\zeta$  is strictly maximized at  $\zeta$  then the corresponding time series $\{\phi\circ T^i\}$
exhibits extreme value statistics corresponding to an iid sequence of random variables with the same distribution function as $\phi$ and with extremal index one.
If  however $\phi$ is strictly maximized at a periodic point  $q$ then the corresponding time-series exhibits extreme value statistics corresponding to an iid sequence of random variables with the same distribution function as $\phi$ but with  extremal index not equal to  one. We give a formula for the extremal index (which depends upon the metric used and the period of $q$). These results imply that return times are Poisson to small balls centered at non-periodic points and compound Poisson for small balls centered at periodic points.
\end{abstract}

\maketitle

\section{Introduction}
Suppose we have a time-series $\{X_i\}$ of real-valued random variables defined on a probability space $(\Omega,P)$. Extreme value theory concerns the statistical
behavior of the derived time series of maxima $M_n:=\max \{X_1,\ldots, X_n\}$. There is a well-developed theory for iid random variables~\cite{LLR83,R87,G78}
and these references also consider the case of dependent random variables under certain assumptions. In the setting of ergodic maps $T: X\to X$  of
a probability space $(X,\mu)$ modeling deterministic physical
phenomena we are concerned with time-series of the form $\{\phi\circ T^i\}$ where $\phi:X \to \R$ is an observation of some regularity, usually H\"older. In the literature, as in
this paper, it is customary to assume that $\phi$ is a function of the distance $d(x,p)$ to a distinguished point $\zeta$ for some metric $d$ so that  $\phi (x)=f(d(x,\zeta))$ for $x\in X$.
For convenience the observation $\phi (x)=-\log d(x,\zeta)$ is often used, but scaling relations translate extreme value results for one functional to another quite easily~\cite{FFT10,HNT12}.
Since the function $\phi\circ T^i (x)$ detects the closer $T^i (x)$ is to $\zeta$,  there is a natural relation between extreme value statistics for the time series $\{\phi\circ T^i\}$
and return time statistics to nested balls about $\zeta$~\cite{C01, FFT12}. In fact, if $(X,d)$ is a Riemannian manifold and $\mu$ is absolutely continuous with
respect to the  volume measure then return time statistics may be deduced from extreme value statistics and vice-versa, so in some sense they are equivalent though proved with
different techniques.

Due to the close relation with return time statistics one expects a difference in the extreme value theory of observations centered at periodic points to those centered at
generic points. We refer to the early paper of Hirata \cite{H93}. What is perhaps surprising is that in many cases of uniformly expanding smooth one dimensional systems  there is a strict dichotomy between the behavior of
observations maximized at periodic points and other non-periodic points. Suppose $\phi$ is maximized at a point $\zeta \in X$ and define a sequence of constants
$u_n$ by the requirement that $\lim_{n\to \infty} n\mu (\phi >u_n)=\tau$, where $\tau \in \R$.  It has been shown~\cite{FP12,FFT12,K12}  that for a variety of  smooth  one-dimensional chaotic systems this implies
that $\lim_{n\to \infty} \mu (M_n \le u_n)=e^{-\vartheta \tau}$ where $0< \vartheta \le 1$ is called the extremal index and roughly measures the clustering of exceedances of the maxima
(more precisely $\vartheta^{-1}$ is the expected number of observed exceedances if one exceedance is observed). It turns out that $\vartheta=1$ if $\zeta$  is not periodic
and $\vartheta<1$ otherwise  (a precise expression for $\vartheta$ may often be  given as  a function of the multiplier derivative of $T$ along the periodic orbit).
For higher dimensional smooth chaotic systems a similar dichotomy is expected (the presence of discontinuities complicates matters~\cite{AFV14}) but so far
has not been established. In this paper we establish such a dichotomy for hyperbolic toral automorphisms, give an expression for the extremal index and find that it also
depends upon the metric used. This has immediate implications for the return time statistics. We have Poissonian return time statistics to non-periodic points and
a compound Poissonian distribution to periodic points (the form of the compound distribution is related to the extremal index, as well). In fact, our results are given in terms of convergence of point processes of rare events to the standard Poisson process, in the non-periodic case, and to a compound Poisson process, in the periodic case, where the multiplicity distribution is also seen to depend on the metric used. We remark that the convergence of the point processes is actually stronger than establishing a Poissonian distributional limit for the number of visits to shrinking target sets.

Earlier results for invertible dynamical systems include work by Denker, Gordin and Sharova, \cite{DGS04}, who proved that, for toral automorphisms, the number of visits to shrinking neighbourhoods of non-periodic points converge (when properly normalised) to a Poisson distribution and  work by Dolgopyat~\cite{D04}, who also proved Poissonian return time statistics for non-periodic  points in uniformly partially hyperbolic dynamical systems
and noted possible generalizations to the case of periodic points. There have been a variety of recent results on extreme value and return time statistics for invertible dynamical systems, both uniformly and non-uniformly hyperbolic, but restricted to observations maximized at generic points~\cite{GHN11,HNT12,CC13, FHN14,Licheng,G10}.

\section{The setting}

Let $L:\R^2 \to \R^2$ be a $2\times2$ matrix with integer entries, with determinant equal to $1$ and without eigenvalues of absolute value $1$. The eigenvalues of $L$ are irrational numbers $\lambda$ and $1/\lambda$ satisfying $|\lambda|>1$; the corresponding eigenspaces, $E^s$ and $E^u$, are lines with irrational slope and the family of lines parallel to these subspaces is invariant under $L$. Moreover, $L$ contracts vectors in $E^s$ and expands in $E^u$. As $L(\mathbb{Z}^2) = \mathbb{Z}^2$, this linear map induces an invertible smooth transformation $T$ on the flat torus $\mathbb{T}^2=\mathbb{R}^2/\mathbb{Z}^2$ through the canonical projection. The subspaces $E^s$ and $E^u$ project on dense curves which intersect densely and are $T$ invariant; that is, the torus is the homoclinic class of the (unique) fixed point of $T$ and a global attractor.

After finding a norm in $\R^2$ that makes $L$ a uniform contraction in $E^s$ and a uniform expansion in $E^u$, we project the Riemannian metric generated by this norm to $\mathbb{T}^2$ and consider the invariant splitting at each point into one-dimensional affine spaces parallel to the subspaces $E^s$ and $E^u$. This way, $T$ becomes an Anosov diffeomorphism: the hyperbolic splitting at each point of $\mathbb{T}^2$ is obtained by translating to that point the eigenspaces of the matrix. Through each point $x\in \mathbb{T}^2$
there is a local stable manifold $W_{\epsilon}^s(x)$  of length $\epsilon$ and similarly a local unstable manifold $W^u_{\epsilon}(x)$. Besides, $T$ is topologically mixing and its periodic points (all saddles) are dense, even though for each $q \in \N$ the set of periodic points with period $q$  is finite.

As the determinant of $L$ is equal to $1$, the Riemannian structure induces a Lebesgue measure $m$ on $\mathbb{T}^2$ which is invariant by $T$. The probability measure $m$ is Bernoulli and it is also the unique equilibrium state of the H\"older map $x \in \mathbb{T}^2 \to \varphi^u(x) = -\log\,|\det\,DT_{E^u(x)}|$, whose pressure is zero. This means that the metric entropy of $m$ is given by $h_m(T)=-\int\,\varphi^u\, dm = \log\,(\lambda)$.

Unless said otherwise, we assume that we are using the Euclidean metric in $\R^2$, which we project to $\mathbb T^2$ and denote by $d(\cdot,\cdot)$. More precisely, if $(x,y)$ denote the local coordinates of the point $z\in \mathbb T^2$ then
\begin{equation}
\label{eq:Euclidean}
d(z,0)=\sqrt{x^2+y^2}.
\end{equation}

Given a non-empty set $A$ and a point $x$, we denote by $\bar A$ the closure of the set $A$ and define the distance of $x$ to $A$ by $\dist(x,A):=\inf\{d(x,y):\; y\in A\}$.

\begin{remark}\label{rem:assumptionA}
Let A be a set homeomorphic to a ball whose boundary is a piecewise differentiable Jordan curve. For $\varepsilon>0$, let $D_\eps:=\{x\in\mathbb T^2:\; \dist(x, \bar A)\leq\eps\}$. Then $m(D_\eps\setminus A)=\eps |\partial A|+o(\eps)$, where $|\cdot|$ denotes the length of the curve $\partial A$.
\end{remark}

Recall that a continuous function $\varphi:\mathbb{T}^2\to \R$ is $\alpha$-H\"older if there are constants $C>0$ and $\alpha>0$ such that $|\varphi(x)-\varphi(y)|\leq C\,d(x,y)^\alpha$. For $\alpha \in\,\,]0,1]$, the space of $\alpha$-H\"older continuous functions with the norm
$$\|\varphi\|_{\alpha} = \|\varphi\|_{C^0} + \sup_{x \neq y}\,\left\{\frac{|\varphi(x)-\varphi(y)|}{d(x,y)^\alpha}\right\}$$
is a Banach space. And, given $k \in \N$ and $\alpha \in \,\, ]0,1[$, any $\alpha$-H\"older continuous function $\varphi$ on the torus admits a Lipschitz approximation $\varphi_k$ such that $\|\varphi-\varphi_k\|_\infty=O(k^{-\frac{\alpha}{1-\alpha}})$. From the computations in \cite[Section 1.26]{B75} one can deduce that the diffeomorphism $T$ has exponential decay of correlations with respect to $\alpha$-H\"older functions, for any $\alpha \in\,\, ]0,1]$. More precisely, there exist constants $C>0$ and $0<\theta<1$ such that, for all $\alpha$-H\"older maps $\phi$ and $\psi$, we have
\begin{equation}\label{Lip_Lip_decay}
\left |\int \phi  \,(\psi\circ T^n)\,  dm -\int \phi \, dm \int \psi \, dm \right| \le C \,\theta^n \,\|\phi\|_{\alpha}\,\, \|\psi \|_{\alpha}.
\end{equation}
And if $\psi$ is constant on local stable leaves then
\begin{equation}\label{Lip_infinity_decay}
\left|\int \phi  \,(\psi\circ T^n)  \,dm -\int \phi \,dm \int \psi \,dm \right|
\le C \,\theta^n \,\|\phi\|_{\alpha} \,\,\|\psi\|_{\infty}.
\end{equation}

Among $C^2$ Anosov diffeomorphisms the subset of the ones that admit no invariant measure absolutely continuous with respect to $m$ is open and dense. However, it is known that any Anosov diffeomorphism of the torus is conjugate to a linear model; see \cite{M74}. The conjugacy $H$ (and its inverse) is H\"older continuous \cite{KH95} and carries over the measure $m$ to a probability measure $\mu=H_*(m)$ on $\mathbb{T}^2$ which is invariant and mixing by the original Anosov diffeomorphism. In particular, the exponential decay of correlations \eqref{Lip_Lip_decay} is still valid for $\mu$, although corresponding to different values of $C$, $\theta$ and $\alpha$. We remark that the conditions we will analyze in the next section, to ensure the existence of an Extreme Value Law, are also invariant under conjugacy.

The starting point of our analysis is a stationary stochastic process $X_0, X_1, X_2, \ldots$ arising from the system described above in the following way. We fix some point $\zeta\in\mathbb T^2$ and let $\varphi: \mathbb T^2\to \R\cup\{+\infty\}$ be given by
\begin{equation}
\label{eq:observable}
\varphi(z)=-\log(d(z,\zeta)),
\end{equation}
where $d$ denotes the projection of the usual Euclidean metric in $\R^2$ to $\mathbb T^2$, as in \eqref{eq:Euclidean}.
We define $X_0=\varphi$ and for all $i\in
\N$
\begin{equation}
\label{eq:SP-def}
X_i=\varphi\circ T^i.
\end{equation}

Since $m$ is an invariant probability measure on $\mathbb T^2$, then the stochastic process $X_0, X_1, X_2,\ldots$ just defined is stationary.

Note that $\{X_0>u\}=B_{\e^{-u}}(\zeta)$, where $B_{\e^{-u}}(\zeta)$ denotes the ball centered at $\zeta$  of radius $\e^{-u}$, in the Euclidean metric. In other words, having an exceedance of an high threshold $u$, at time $j$, means that, at the $j$th iterate, the orbit enters the ball around $\zeta$ of radius $\e^{-u}$.

We remark that we could use a function different from $-\log$ in the definition of $\varphi$. In fact, any not too irregular function, achieving a global maximum at $0$ and invertible in a vicinity of $0$, would work. See \cite[Section~1.1]{FFT10} for conditions on this function. For definiteness, here we will use $-\log$ which puts us on the Type 1 or Gumbel domain of attraction for the maxima. Our goal is to show the existence of an EVL and for that purpose the analysis does not change much from one example to another.

In order to obtain a non-degenerate limit for the distribution of $M_n:=\max\{X_0, \ldots, X_{n-1}\}$, one usually uses normalizing sequences $(u_n)_{n\in\N}$ such that the average number of exceedances among the first $n$ observations is some constant $\tau\geq 0$, when $n$ is large enough, \ie the sequence $(u_n)_{n\in\N}$ satisfies:
\begin{equation}
\label{eq:un}
n\, m(X_0>u_n)\xrightarrow[n\to\infty]{}\tau\geq 0.
\end{equation}

In the case here, if we take $u_n=\frac 12 \log(\pi n)-\frac 12 \log \tau$, then conditions \eqref{eq:un} is satisfied.

We are now in condition of stating our first result which asserts a dichotomy regarding the possible limit distributions for $M_n$ depending on the periodicity or aperiodicity of the point $\zeta$.

\begin{theorem}
\label{thm:dichotomy1}
Let  $T:\mathbb T^2\to \mathbb T^2$ be an Anosov linear diffeomorphism. Fix $\zeta\in \mathbb T^2$ and define the stochastic process $X_0, X_1, X_2, \ldots$ as in \eqref{eq:SP-def}. Let  $(u_n)_{n\in\N}$ be a sequence satisfying \eqref{eq:un} for some $\tau\geq 0$. Then
\begin{enumerate}

\item if $\zeta$ is not periodic, we have $\lim_{n\to\infty} m(M_n\leq u_n)=\e^{-\tau}$.

\item if $\zeta$ is periodic of prime period $q$, we have $\lim_{n\to\infty} m(M_n\leq u_n)=\e^{-\vartheta\tau}$, where $\vartheta=\frac2\pi  \left(\arcsin\,\frac{|\lambda|^q}{\sqrt{|\lambda|^{2q}+1}}-\arcsin\,\frac{1}{\sqrt{|\lambda|^{2q}+1}}\right)$.

\end{enumerate}

\end{theorem}

In contrast to the examples of uniformly expanding systems (with no contracting directions), the metric used plays an important role in the computation of the EI. To illustrate this, let $e^u, e^s$ denote unit eigenvectors associated to the eigenvalues $\lambda$ and $1/\lambda$, respectively; thus $\{e^u, e^s\}$ forms a normal basis of $\R^2$. Let $z=(x^u,x^s)$, where $(x^u ,x^s)$ stands for the coordinates of $z$ in this basis. Define the metric $d^*$ by
\begin{equation}
\label{eq:adjusted-metric}
d^*(z,0)=\max\{|x^u|,|x^s|\}
\end{equation}
and then project it to $\mathbb T^2$.

\begin{theorem}
\label{thm:dichotomy2}
Let  $T:\mathbb T^2\to \mathbb T^2$ be an Anosov linear diffeomorphism. Fix $\zeta\in \mathbb T^2$ and define the stochastic process $X_0, X_1, X_2, \ldots$ as in \eqref{eq:SP-def} for the observable $\varphi$ as in \eqref{eq:observable} but evaluated with the metric $d^*$ instead. Let  $(u_n)_{n\in\N}$ be a sequence satisfying \eqref{eq:un} for some $\tau\geq 0$. Then
\begin{enumerate}

\item if $\zeta$ is not periodic, we have $\lim_{n\to\infty} m(M_n\leq u_n)=\e^{-\tau}$.

\item if $\zeta$ is periodic of prime period $q$, we have $\lim_{n\to\infty} m(M_n\leq u_n)=\e^{-\vartheta\tau}$, where $\vartheta=1-|\lambda|^{-q}$.

\end{enumerate}

\end{theorem}

\section{The approach and dependence conditions }

In order to prove the existence of EVLs in a dynamical systems context, there are a couple of conditions on the dependence structure of the stochastic process that if verified allow us to obtain such distributional limits. These conditions are motivated by the conditions $D(u_n)$ and $D'(u_n)$ of Leadbetter but were adapted to the dynamical setting and further developed both in the absence of clustering,  such as in \cite{C01, FF08a, HNT12}, and in the presence of clustering in \cite{FFT12}. Very recently, in \cite{FFT14}, the authors provided some more general conditions, called $\D(u_n)$ and $\D_q'(u_n)$, which subsumed  the previous ones and allowed them  to address both the presence ($q\geq1$) and the absence ($q=0$) of clustering. To distinguish these conditions the authors used a Cyrillic D to denote them. We recall these conditions here.

Given a sequence $(u_n)_{n \in \N}$ of real numbers satisfying \eqref{eq:un} and $q\in\N_0$, set
$$A_n^{(q)}:=\{X_0>u_n,X_1\leq u_n,\ldots, X_q\leq u_n\}.
$$

For $s,\ell\in\N$ and an event $B$, let
\begin{equation}
\label{eq:W-def}
\mathscr W_{s,\ell}(B)=\bigcap_{i=s}^{s+\ell-1} T^{-i}(B^c).
\end{equation}

\begin{condition}[$\D_q(u_n)$]\label{cond:D} We say that $\D(u_n)$ holds for the sequence $X_0,X_1,\ldots$ if, for every  $\ell,t,n\in\N$
\begin{equation}\label{eq:D1}
\left|\p\left(\A_n\cap
 \mathscr W_{t,\ell}\left(\A_n\right) \right)-\p\left(\A_n\right)
  \p\left(\mathscr W_{0,\ell}\left(\A_n\right)\right)\right|\leq \gamma(q,n,t),
\end{equation}
where $\gamma(q,n,t)$ is decreasing in $t$ and  there exists a sequence $(t_n)_{n\in\N}$ such that $t_n=o(n)$ and
$n\gamma(q,n,t_n)\to0$ when $n\rightarrow\infty$.
\end{condition}

For some fixed $q\in\N_0$, consider the sequence $(t_n)_{n\in\N}$ given by condition $\D_q(u_n)$ and let $(k_n)_{n\in\N}$ be another sequence of integers such that
\begin{equation}
\label{eq:kn-sequence}
k_n\to\infty\quad \mbox{and}\quad  k_n t_n = o(n).
\end{equation}

\begin{condition}[$\D'_q(u_n)$]\label{cond:D'q} We say that $\D'_q(u_n)$
holds for the sequence $X_0,X_1,\ldots$ if there exists a sequence $(k_n)_{n\in\N}$ satisfying \eqref{eq:kn-sequence} and such that
\begin{equation}
\label{eq:D'rho-un}
\lim_{n\rightarrow\infty}\,n\sum_{j=1}^{\lfloor n/k_n\rfloor}\p\left( \A_n\cap T^{-j}\left(\A_n\right)
\right)=0.
\end{equation}
\end{condition}

We note that, when $q=0$, condition $\D'_q(u_n)$ corresponds to condition $D'(u_n)$ from \cite{L73}.

Now let
\begin{equation}
\label{eq:OBrien-EI}
\vartheta=\lim_{n\to\infty}\vartheta_n=\lim_{n\to\infty}\frac{\p(\A_n)}{\p(U_n)}.
\end{equation}

From \cite[Corollary~2.4]{FFT14}, it follows that if the stochastic process $X_0, X_1,\ldots$ satisfies both conditions $\D_q(u_n)$ and $\D'_q(u_n)$, the limit in \eqref{eq:OBrien-EI} exists and 
$$\lim_{n\to\infty}\p(M_n\leq u_n)= \e^{-\vartheta\tau}.$$

Hence, the strategy to prove Theorems~\ref{thm:dichotomy1} and \ref{thm:dichotomy2} is to show that conditions $\D_0(u_n)$ and $\D'_0(u_n)$ hold when $\zeta$ is a non periodic point, and that conditions $\D_q(u_n)$ and $\D'_q(u_n)$ hold when $\zeta$ is a periodic point of prime period $q$. In fact, we only check these conditions for the usual Euclidean metric as in the context of Theorem~\ref{thm:dichotomy1}, which is technically harder, leaving the necessary adjustments when dealing with the adapted metric $d^*$ for the reader.

\section{Checking condition $\D_q(u_n)$}
\label{sec:D}

Let $D$ be a set homeomorphic to a ball whose boundary is piecewise smooth, so $m(\partial D)=0$, and define
\[
 H_{k}(D) = \left\{x \in D:T^{k}(W^s_1(x))\cap \partial D\ne\emptyset\right\}.
 \]

\begin{proposition}\label{prop:annulus1}
There exist constants $C>0$ and $0<\tau_1<1$ such that, for all $k$,
\begin{equation}\label{annulus}
 m(H_{ k}(D))\le  C \tau_1^{k}.
\end{equation}
\end{proposition}

\begin{proof}
 As a consequence of the uniform contraction within local stable manifolds, 
there exists 
$C_1>0$ such that $\dist\,(T^n(x), T^n(y))\le C_1/|\lambda|^n$ for all
 $y\in W^s_1(x).$ In particular, this implies that $|T^k(W_1^s(x))|\le C_1/|\lambda|^k$. Therefore, for every  $x\in H_{k}(D)$,
 the leaf $T^k(W^s_1(x))$ lies in an annulus of width $2 /|\lambda|^k$ around $\partial D$. By Remark~\ref{rem:assumptionA} and the invariance of $m$, the result follows with $\tau_1=1/|\lambda|$.
\end{proof}


\begin{lemma}\label{lemma:dun-prelim}
Suppose
 $\Phi:M\to \R$ is a Lipschitz map and $\Psi$ is the indicator function
 \[
\Psi:= \I_{\mathscr W_{0,\ell}\left(\A_n\right)}.
 \]
Then for all $j\geq 0$
\begin{equation}
 \left|\int\Phi\,(\Psi\circ T^j)\, \text{d}m - \int\Phi\text{d}m\int\Psi\text{d}m\right|\le \C\,\left(\|\Phi\|_\infty \,\,\tau_1^{\floor{j/2}}+\|\Phi\|_{\text{Lip}}\,\,\theta^{\floor{j/2}}\right).
\end{equation}
\end{lemma}

\begin{proof}

We choose for reference the fixed point $\zeta$ of $T$ and its local unstable manifold $\tilde\gamma^u:=W_1^u(\zeta)$. By the hyperbolic product structure, each local stable manifold $W^s_1(x)$ intersects $\tilde\gamma^u$ in a unique point $\hat{x}$. Therefore, for every map $\Psi$ we define the function $\overline\Psi(x):=\Psi(\hat x)$. Also, for every $i\in\N$,  let $\Psi_i=\Psi\circ T^i$. We note that $\overline \Psi_i$ is constant along stable manifolds 
and that the set of points where $\overline \Psi_i\neq\Psi_i$ is, by definition, the set of $x$  for which there
exist $r \in \mathbb{N}_0$ and $x_1,x_2$ on the same local stable manifold as $T^r (x)$ such that
\[
x_1\in\mathscr W_{i,\ell}\left(\A_n\right)
\]
 but
\[
x_2 \notin\mathscr W_{i,\ell}\left(\A_n\right).
\]
Moreover, this set is contained in $\cup_{k =i}^{i+\ell-1}H_k(\A_n)$. If
we let $i\ge\floor{j/2}$ then by Proposition~\ref{prop:annulus1}
we  have
\[m(\set{\Psi_{i}\neq\overline\Psi_{i}}) \le
\sum_{k=\floor{j/2}}^{\infty} m(H_k(\A_n))\le \C \,\tau_1^{\floor{j/2}}.\]
By \eqref{Lip_infinity_decay}, we also get
\[
 \abs{\int\Phi\,(\,\overline\Psi_{\floor{j/2}}\circ T^{j -\floor{j/2}})\,dm - \int\Phi dm\int \overline\Psi_{\floor{j/2}}dm}\le \C\,\|\Phi\|_{\text{Lip}}\,\,\|\bar\Psi\|_{\infty}\,\,\theta^{\floor{j/2}}.
\]
Using the identity
$$\int \phi\,(\psi\circ T)-\int \phi \int \psi = \int\phi \,(\psi\circ T-\bar{\psi}\circ T)+
\int \phi\,(\bar{\psi}\circ T) -\int \phi \int \bar{\psi} +\int \phi \int \bar{\psi} -\int \phi \int \psi$$
we obtain
\begin{align}
\Big|\int \Phi\,\left(\Psi_{\floor{j/2}}\circ T^{j -\floor{j/2}}\right)\,dm &- \int\Phi dm\int \Psi_{\floor{j/2}}\,dm\Big|\nonumber\\
%
&\le \abs{\int \Phi\, \left((\Psi_{\floor{j/2}} - \overline\Psi_{\floor{j/2}})\circ T^{j - \floor{j/2}}\right)\,dm}  +\C \,\|\Phi\|_{\text{Lip}}\,\,\theta^{\floor{j/2}} \nonumber\\
&\quad + \abs{\int\Phi \,dm \int\left(\overline\Psi_{\floor{j/2}} - \Psi_{\floor{j/2}}\right)}
\,dm\nonumber\\
&\le \C\left(2\,\|\Phi\|_\infty \,\,m\set{\overline\Psi_{\floor{j/2}}\neq\Psi_{\floor{j/2}}}+\|\Phi\|_{\text{Lip}}\,\,\theta^{\floor{j/2}}\right)\nonumber\\
&\le \C\left(\|\Phi\|_\infty \,\, \tau_1^{\floor{j/2}}+\|\Phi\|_{\text{Lip}}\,\,\theta^{\floor{j/2}}\right).
\end{align}
We complete the proof by observing that $\int\Psi \,dm = \int\Psi_{\floor{j/2}} \,dm$, due to the $T$-invariance of $m$ and the fact that
 $\Psi_{\floor{j/2}}\circ T^{j -\floor{j/2}} = \Psi_{j} = \Psi\circ T^{j}.$
\end{proof}

To prove condition $\D(u_n)$, we will approximate the characteristic function of the set $\A_n$ by a suitable Lipschitz function. However, this Lipschitz function will decrease sharply to zero near the boundary of the set $\A_n$. As the estimate in Lemma~\ref{lemma:dun-prelim} involves the Lipschitz norm, we need to bound its increase as we approach $\partial \A_n$.

Let $A_n=\A_n$ and $D_n:=\set{x\in \A_n:\;\dist\left(x, \overline{A_n^c}\right)\geq n^{-2}},$
 where $\bar A_n^c$ denotes the closure of the complement of the set $A_n$.
Define $\Phi_n:\X\to\R$ as
\begin{equation}
\label{eq:Lip-approximation}
\Phi_n(x)=\begin{cases}
  0&\text{if $x\notin A_n$}\\
  \frac{\dist(x,A_n^c)}{\dist(x,A_n^c)+
  \dist(x,D_n)}&
  \text{if $x\in A_n\setminus D_n$}\\
  1& \text{if $x\in D_n$}
\end{cases}.
\end{equation}
Note that $\Phi_n$ is Lipschitz continuous with Lipschitz constant given by $n^2$. 
Moreover, as observed in  Remark~\ref{rem:assumptionA}, we have $\|\Phi_n-\I_{A_n}\|_{L^1(m)}\leq C/n^2$ for some constant $C$.

It follows that
\begin{align}
\Big|\int \I_{\A_n}\,\left(\Psi_{\floor{j/2}}\circ T^{j -\floor{j/2}}\right) &\,dm - m(\A_n)\int \Psi dm\Big|\nonumber\\
&\le \abs{\int \left(\I_{\A_n} - \Phi_n\right)\Psi_{\floor{j/2}}\,dm}+\C \left(\|\Phi_n \|_\infty \,\,j^2\,\,\tau_1^{\floor{j/4}}+\|\Phi_n \|_{\text{Lip}}\,\,\theta^{\floor{j/2}}\right)\nonumber\\
&\quad+ \abs{\int \left(\I_{\A_n} - \Phi_n \right)\,dm \int \Psi_{\floor{j/2}}\,dm},
\end{align}
and consequently
\[
 \abs{m\left(\A_n\cap \mathscr W_{j,\ell}(\A_n)\right) - m(\A_n)\,m\left(\mathscr W_{0,\ell}(\A_n)\right)}\le \gamma(n,j)
\] where
\[
 \gamma(n,j) = \C\,\left(n^{-2}+ n^{2} \,\theta_1^{\floor{j/2}}\right)
\] and
\[
\theta_1 = \max\,\set{\tau_1, \theta}.\]
Thus if, for instance, $j=t_n=(\log n)^{5}$, then $n\gamma(n, t_n)\to 0$ as $n\to\infty.$  Note that we have considerable
freedom of choice of $j$ in order to ensure that the previous limit is zero; taking into account the possible applications, we chose $t_n=(\log n)^{5}$.

\section{Checking condition $\D'_q(u_n)$}
\label{sec:D'q}

Before checking the condition $\D'_q(u_n)$, we observe that we only need to consider the sum in \eqref{eq:D'rho-un} up to $t_n=(\log n)^5$ since, by the exponential decay of correlations, the sum of the remaining terms goes to $0$. More precisely, we use decay of correlations to show that
\begin{equation}\label{limiting.short.terms}
 \lim_{n\rightarrow \infty} n \sum_{j=\log^5n}^{\floor{n/k_n}} m (\A_n\cap T^{-j}\A_n)=0.
\end{equation}

Let $\Phi_n$ be a suitable Lipschitz approximation of $\I_{\A_n}$, defined as in \ref{eq:Lip-approximation}. Then
\begin{align*}
\left|\int 1_{\A_n}(1_{\A_n} \circ T^j)\,dm - \left(\int 1_{\A_n}~dm\right)^2 \right| &\le \left |\int \Phi_n(\Phi_n \circ T^j)\,dm -\left(\int \Phi_n~dm\right)^2\right|\\
&\quad+\left|\left(\int \Phi_n~dm\right)^2- \left(\int 1_{\A_n}~dm\right)^2\right|\\
&\quad+ \left| \int 1_{\A_n}(1_{\A_n} \circ T^j)\,dm -\int \Phi_n(\Phi_n\circ T^j)\,dm \right|.
\end{align*}
If  $ (\log n)^5\le j \le \floor{n/k_n}$, then, due to the decay of correlations, the first term is upper bounded by
$$
\left|\int \Phi_n (\Phi_n \circ T^j)\,dm -\left(\int \Phi_n~dm\right)^2 \right |\le C \,n^{2}\, \theta^{j}\le C\,{n^{-2}}
$$
if  $n$  is sufficiently large. Recalling Remark~\ref{rem:assumptionA}, for the second term we obtain, for $n$ big enough,
$$
\left|\left(\int \Phi_n \,dm\right)^2-\left(\int 1_{\A_n}~dm\right)^2\right|\le C \,m(\A_n\setminus D_n)\le C \, n^{-2}.
$$
Similarly, we estimate the third term as follows
 $$
\left|\int \Phi_n (\Phi_n \circ T^j)\,dm - \int 1_{\A_n} (1_{\A_n}\circ T^j)\,dm\right|
\le 2\,m(\A_n\setminus D_n) \le C\, n^{-2}.
$$
 Hence equation~\eqref{limiting.short.terms} is satisfied.

\subsection{The periodic case: checking $\D'_q(u_n)$ for $q>0$}
\label{subsec:D'q}

Suppose $\zeta$ is a periodic point of minimal period $q$  for the map $T$. If $G=T^q$, then $\zeta$ is a fixed point for the Anosov diffeomorphism $G$, which has eigenvalues $\frac{1}{\lambda^q}$ and ${\lambda}^q$, with
$\frac{1}{|\lambda|^q} < 1 < {|\lambda|}^q$, where $\frac{1}{\lambda}$ and $\lambda$ are the eigenvalues of the original map $T$.

\begin{figure}
\centering
\begin{subfigure}{.5\textwidth}
  \centering
  \includegraphics[width=\linewidth]{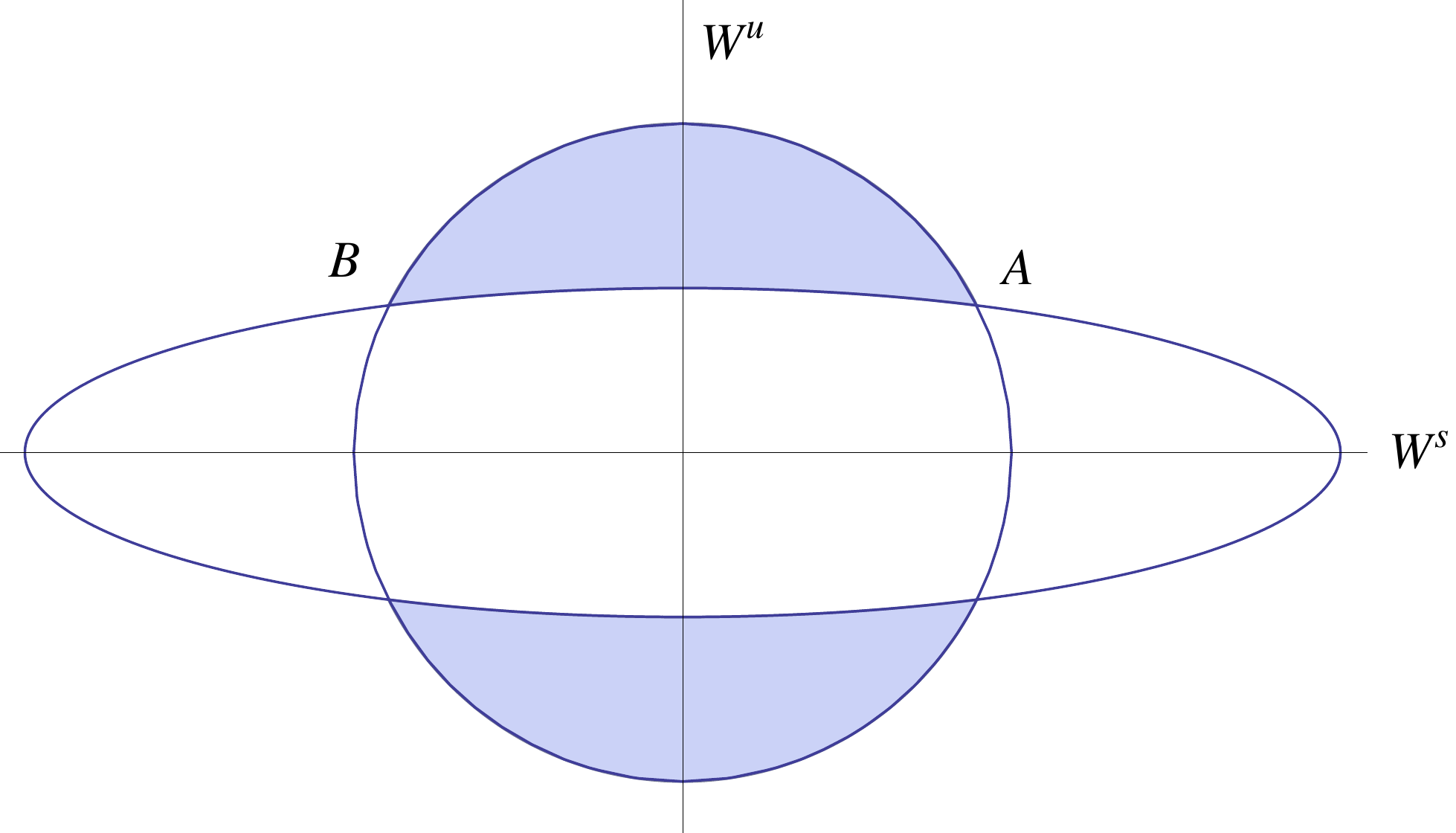}
  \caption{Usual Euclidean metric $d$}
  \label{fig:sub1}
\end{subfigure}%
\begin{subfigure}{.5\textwidth}
  \centering
  \includegraphics[width=\linewidth]{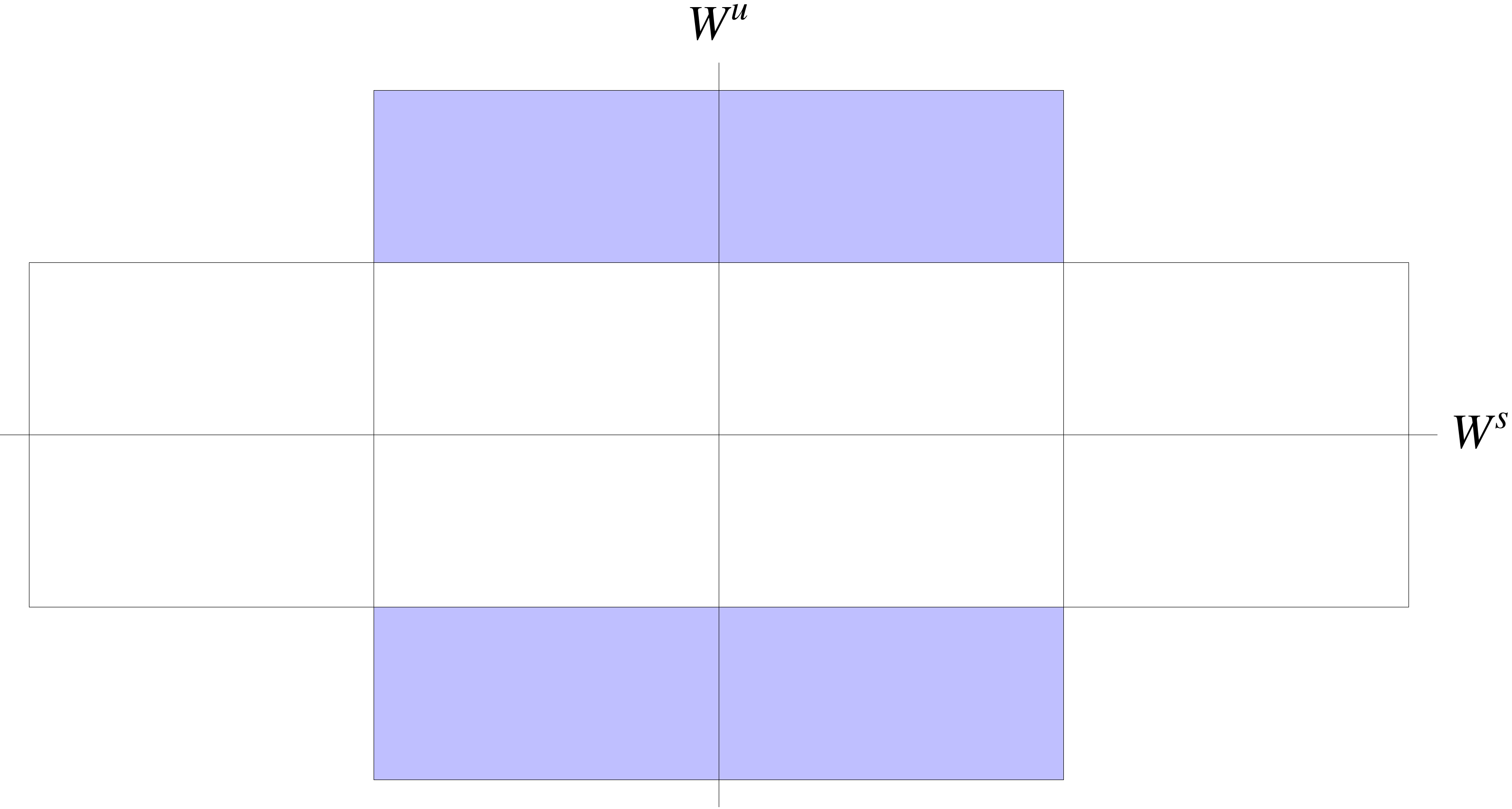}
  \caption{Adapted metric $d^*$}
  \label{fig:sub2}
\end{subfigure}
\caption{The sets $\A_n$ for $q>0$.}
\label{fig:An}
\end{figure}

Note that $\A_n$ corresponds to the shaded region in Figure~\ref{fig:sub1}, the ball represents the set $\{X_0>u_n\}$, while the ellipse is the pre-image $T^{-q}(\{X_0>u_n\})$. To compute the diameter of $\A_n$, we just need to find the distance between the points $A$ and $B$ on the intersection of the ball with the ellipse. If we use local coordinates $x,y$ for the unstable and stable local manifolds through $\zeta$, respectively, then the ball's equation is $x^2+y^2=s_n^2$ and the ellipse's equation is $|\lambda|^{2q}x^2+|\lambda|^{-2q}y^2=s_n^2$, where $s_n^2\sim\tau/(\pi n)$. Then the $x$ coordinate of the points $A$ and $B$ is $\frac{s_n}{\sqrt{|\lambda|^{2q}+1}}$, while the $y$ coordinates are $\frac{|\lambda|^q\,s_n}{\sqrt{|\lambda|^{2q}+1}}$ and its symmetric counterpart, respectively. This means that the diameter of $\A_n$ is $\frac{2\,|\lambda|^q\,s_n}{\sqrt{|\lambda|^{2q}+1}}$.

Note that, under $G^{-j}$, the local stable manifolds in a small ball around $\zeta$ expand uniformly and the local unstable manifolds contract uniformly. By construction of the sets $\A_n$ it follows that $G^{-j}(\A_n)\cap \A_n=\emptyset$, while the diameter of $G^{-j}(\A_n)$, which is measured along the stable direction, is less than $1$. Hence, $G^{-j}(\A_n)\cap \A_n=\emptyset$ for all $j=1,\ldots, g(n)$, where
$$g(n)=\left\lfloor\frac{\log n+\log(|\lambda|^{2q}+1)-2\log(2\,|\lambda|^q(\tau/\pi)^{1/2})}{2q\,\log |\lambda|}\right\rfloor.$$
Since $\zeta$ is periodic of period $q$,  for large $n$  we have $T^{-j}\A_n\cap \A_n=\emptyset$ for $j=1,\ldots,qg(n)$. Thus, 
\begin{equation}
\label{eq:first-parcels}
\sum_{j=1}^{qg(n)} m(\A_n\cap T^{-j}(\A_n))=0.
\end{equation}

Let $x\in \A_n\cap W^u_\zeta$ and $\gamma_x:=W^s_x\cap \A_n$. For $j>qg(n)$, the curve $T^{-j}(\gamma_x)$ may be sufficiently long to wrap around the torus. Yet, if this is the case then
$$|T^{-j}(\gamma_x)\cap \A_n|\leq \frac C{\sqrt n}\,|T^{-j}(\gamma_x)|.$$
Indeed, the length of $\A_n$ in the stable direction is bounded above by $C_1 \,\frac{1}{\sqrt{n}}$, for some constant $C_1$ independent of the local stable manifold. As $T$ restricted to local stable manifolds has bounded distortion (in fact its derivative at any point is even constant, equal to $L$), it follows that $|\{y:y\in \gamma_x,\; T^{-j}(y)\in \A_n\}|\leq \frac C{\sqrt n}\,|\gamma_x|$. Moreover, since the 2-dimensional Lebesgue measure $m$ is invariant by $T$, we have
\begin{align*}
m\left(\A_n\cap T^{-j} (\A_n)\right) & = m\left(\A_n\cap T^j (\A_n)\right) =\int_{\A_n\cap W_\zeta^u} \,|\gamma_x\cap T^j (\A_n)| \, dm_x \\
&\leq \frac C {\sqrt n}\int_{\A_n\cap W_\zeta^u}  \,|\gamma_x|\, dm_x \leq \frac{C}{\sqrt n} \,m(\A_n).
\end{align*}

This implies that 
\begin{equation}
\label{eq:midle-parcels}
n \sum_{j=qg(n)+1}^{log^5 n} m\left(\A_n\cap T^{-j} (\A_n)\right) = O\left(n(\log n)^5 \,n^{-1/2}\, m(\A_n)\right) = O(n^{-1/2}\log^5 n) \xrightarrow[n\to\infty]{} 0
\end{equation}
which, together with \eqref{limiting.short.terms} and \eqref{eq:first-parcels}, yields \eqref{eq:D'rho-un}.

\subsection{The non-periodic case: checking $\D'_q(u_n)$ for $q=0$}
\label{subsec:D'}

Let us consider now a non-periodic point $\zeta$. 
Note that, in this case,  $\A_n=U_n=\{X_0>u_n\}$ is a ball centered at $\zeta$ such that $m(U_n)\sim\tau/n$. We split the sum in \eqref{eq:D'rho-un} into three parts:
\begin{align}\label{eq:sum.non.periodic}
n\sum_{j=0}^{\floor{n/k_n}} m(U_n\cap T^{-j}(U_n))&=n\sum_{j=0}^{g(n)} m(U_n\cap T^{-j}(U_n)) +n\sum_{j=g(n)+1}^{\log^5 n}m(U_n\cap T^{-j}(U_n))\nonumber\\
&\quad+n \sum_{j=\log^5 n+1}^{n/k_n}m(U_n\cap T^{-j}(U_n)),
\end{align}
where $g(n)=\floor{\frac{\log n-\log\pi}{2\log |\lambda|}}$ is such that, for $j>g(n)$, the set $T^{-j}(U_n)$ is sufficiently stretched along the stable direction to start wrapping around the torus. The third summand on the right of \eqref{eq:sum.non.periodic} goes to $0$ on account of \eqref{limiting.short.terms}. We proceed considering the other summands, starting by the first. Contrary to the periodic setting, we now have to deal with the possibility of  $U_n\cap T^{-j}(U_n)\neq \emptyset$ for $j\leq g(n)$. Let $R_n=\min\{j\in\mathbb N: U_n\cap T^{-j}(U_n)\neq \emptyset\}$. By continuity of $T$ and the fact that $\displaystyle \cap_n U_n=\{\zeta\}$, we have that $R_n\to\infty$ as $n\to\infty$.

For $j\leq g(n)$, observe that $T^{-j}(U_n)$  is an ellipse centered at $T^{-j}(\zeta)$, stretched along the stable direction by a factor $|\lambda|^j$ and contracted on the unstable direction by $1/|\lambda|^j$. Also, notice that, as we are dealing with a projection on the torus of a linear map in $\R^2$, the stable and unstable directions are always lined up regardless of the position of $T^{-j}(\zeta)$. Now, since the set $T^{-j}(U_n)$ is not long enough (on the stable direction) to wrap around the torus, the diameter of $U_n$ is $2 s_n$, where $s_n\sim \sqrt\tau/\sqrt {\pi n}$. The length of $T^{-j}(U_n)$ along the stable direction (which coincides with the diameter of $T^{-j}(U_n)$)  is $|\lambda|^j \,s_n$; and the length of  $T^{-j}(U_n)$ on the unstable direction is $|\lambda|^{-j}\, s_n$. So, by the geometry of the sets, the area of the intersection $U_n\cap T^{-j}(U_n)$ occupies a portion of $T^{-j}(U_n)$ corresponding at most to a constant times $|\lambda|^{-j}$ of its area. That is, there exists some $C>0$ independent of $j$ such that
$$
m(U_n\cap T^{-j}(U_n))\leq C \,|\lambda|^{-j} \,m(T^{-j}(U_n))= C\, |\lambda|^{-j} \,m(U_n).
$$
In the case of the Euclidean metric we can take $C\approx2/\pi.$ Then it follows that
$$
n\sum_{j=0}^{g(n)} m(U_n\cap T^{-j}(U_n))=n\sum_{j=R_n}^{g(n)} |\lambda|^{-j}\,\tau/n\leq C \,\tau \,|\lambda|^{-R_n}\xrightarrow[n\to\infty]{}0.
$$

We are left to estimate the second summand of \eqref{eq:sum.non.periodic}. Let $x\in U_n\cap W^u_\zeta$ and $\gamma_x:=W^s_x\cap U_n$. For $j>g(n)$, the curve $T^{-j}(\gamma_x)$ may be sufficiently long to wrap around the torus. Set
$$\tilde U_n^j:=\bigcup_{\set{x\,\in\, U_n\cap W^u_\zeta\,:\, |T^{-j}(\gamma_x)|>1/2}} \gamma_x.$$
Note that for the usual Euclidean metric we have $|W^u_\zeta\cap U_n\setminus\tilde U_n^j|\leq 2 \sqrt{s_n^2-|\lambda|^{-2j}/16}.$ Since for $x\in\tilde U_n^j$, we have $|T^{-j}(\gamma_x)|>1/2$ and the diameter of $U_n$ is less than $2s_n$, then there exists $C>0$ such that, for all $j>g(n)$ and all $x\in\tilde U_n^j$, we have
$$|T^{-j}(\gamma_x)\cap U_n|\leq C\,s_n\,|T^{-j}(\gamma_x)|.$$

As $T$ restricted to local stable manifolds has bounded distortion (since the derivative is constant), it follows that $|\{y:y\in \gamma_x,\; T^{-j}(y)\in U_n\}|\leq  C \,s_n\,|\gamma_x|$. Moreover, by the $T$-invariance of $m$ we have
\begin{align}
m(U_n\cap T^{-j} (\tilde U_n^j))&= m(\tilde U_n^j\cap T^j (U_n))=\int_{\tilde U_n^j\cap W_\zeta^u} \,|\gamma_x\cap T^j (\A_n)| \, dm_x\nonumber\\
&\leq C \,s_n \int_{\tilde U_n^j \cap W_\zeta^u} \, |\gamma_x|\, dm_x \leq C\, s_n\,  m(\tilde U_n^j).
\label{eq:stretched-estimate}
\end{align}

Now, we focus on $T^{-j}(U_n \setminus \tilde U_n^j)$ which is a made of two connected components with diameter $1/2$ along the stable direction and with length $|\lambda|^{-j}\, s_n\Big(1-\sqrt{1-|\lambda|^{-2j}/(16 {s_n}^2)}\Big)$. As before, the geometry dictates that the area of the intersection $U_n\cap T^{-j}(U_n\setminus\tilde U_n^j)$ occupies  a portion of $T^{-j}(U_n\setminus\tilde U_n^j)$ that corresponds at most to a constant times $s_n$ of its area. Indeed, there exists $C>0$ such that,
\begin{equation}
\label{eq:small-estimate}
m(U_n\cap T^{-j}(U_n\setminus\tilde U_n^j))\leq  C\, s_n \,m(T^{-j}(U_n\setminus\tilde U_n^j))= C\, s_n \,m(U_n\setminus\tilde U_n^j).
\end{equation}

Joining \eqref{eq:stretched-estimate}, \eqref{eq:small-estimate} and recalling that $s_n\sim \sqrt\tau/\sqrt {\pi n}$, we deduce that
\begin{equation}
\label{eq:midle-parcels}
n \sum_{j=g(n)+1}^{log^5 n} m(U_n\cap T^{-j} (U_n)) = O(n(\log n)^5 n^{-1/2} m(U_n) )=O(n^{-1/2}\log^5 n) \xrightarrow[n\to\infty]{} 0.
\end{equation}

\section{Computing the Extremal Index}

We recall that the Extremal Index $\vartheta$ is given by the limit in \eqref{eq:OBrien-EI}.

In the case of a non-periodic point $\zeta$, we have $q=0$ and the computation is trivial because $\A_n=U_n$, hence $\vartheta=1$. 

If $\zeta$ is a periodic point of prime period $q$, we need to compute the area of the shaded regions in Figure~\ref{fig:sub1}. Recalling that the second coordinate of the point $A$ of the picture is $y_A=\frac{|\lambda|^q\, s_n}{\sqrt{|\lambda|^{2q}+1}}$ and using the symmetry of the region, we have
\begin{align}
\label{eq:Aq}
m(\A_n)&=4\int_0^{y_A}\sqrt{s_n^2-y^2}-|\lambda|^{-q}\sqrt{s_n^2-y^2\,|\lambda|^{-2q}}\,dy \\
&=2 \,s_n^2 \left(\arcsin\,\frac{|\lambda|^q}{\sqrt{|\lambda|^{2q}+1}}-\arcsin\,\frac{1}{\sqrt{|\lambda|^{2q}+1}}\right)\nonumber.
\end{align}
Thus, in this case,
\begin{equation}
\label{eq:EI-computed}
\vartheta_q=\frac2\pi  \left(\arcsin\,\frac{|\lambda|^q}{\sqrt{|\lambda|^{2q}+1}}-\arcsin\,\frac{1}{\sqrt{|\lambda|^{2q}+1}}\right).
\end{equation}

\begin{remark}
\label{rem:EI-metric}
It is interesting to observe that the metric used affects the value of the Extremal Index. In fact, if we had used the adapted metric, so that its balls would correspond to squares with sides lined up with the stable and unstable directions, then the EI would be equal to $1-|\lambda|^{-q}$, which is unsurprisingly more in tune with the one-dimensional setting. See Figure~\ref{fig:sub2}. Notice also that, in both metrics, the extremal index approaches the non-periodic value as the period goes to $\infty$, that is, $\lim_{q \rightarrow +\infty}\,\vartheta_q=1$.
\end{remark}

\section{Conclusion}
For $C^2$ expanding maps of the interval a dichotomy exists between the extremal behavior of periodic points and non-periodic points~\cite{FP12,FFT13,K12}. In this paper we demonstrated that the same dichotomy holds for hyperbolic total automorphisms.  The extreme value statistics of  observations maximized at generic points in a variety of
non-uniformly hyperbolic systems is known to be the same as that of a sequence of iid random variables with the same distribution function~\cite{FFT10,FFT12,HNT12,GHN11}
but the statistics of functions maximized at periodic points has not yet been established. We believe that it will be the same as in the iid case for smooth non-uniformly hyperbolic
systems but expect mixed distributions (as in~\cite{AFV14})  for non-uniformly  systems with discontinuities.

\section{Rare Events Point Processes}

In order to make the presentation more tractable we started by stating the dichotomy in terms of the possible limit distributions for the maxima when the centers are chosen to be either non-periodic or periodic points. However, we can enrich the results without much extra work by considering the convergence of point processes of exceedances or rare events.

If we consider multiple exceedances we are led to point processes of rare events counting the number of exceedances in a certain time frame. For every $A\subset\R$ we define
\[
\nn_u(A):=\sum_{i\in A\cap\N_0}\I_{X_i>u}.
\]
In the particular case where $A=[a,b)$ we simply write
$\nn_{u,a}^b:=\nn_u([a,b)).$
Observe that $\nn_{u,0}^n$ counts the number of exceedances amongst the first $n$ observations of the process $X_0,X_1,\ldots,X_n$ or, in other words, the number of entrances in $U(u)$ up to time $n$. Also, note that
\begin{equation}
\label{eq:rel-HTS-EVL-pp}
\{\nn_{u,0}^n=0\}=\{M_n\leq u\}
\end{equation}

In order to define a point process that captures the essence of an EVL and HTS through \eqref{eq:rel-HTS-EVL-pp}, we need to re-scale time using the factor $v:=1/\p(X>u)$ given by Kac's Theorem. Let $\S$ denote the semi-ring of subsets of  $\R_0^+$ whose elements
are intervals of the type $[a,b)$, for $a,b\in\R_0^+$. Denote by $\RR$ the ring generated by $\S$; recall that for every $J\in\RR$
there are $k\in\N$ and $k$ intervals $I_1,\ldots,I_k\in\S$, say $I_j=[a_j,b_j)$ with $a_j,b_j\in\R_0^+$, such that
$J=\cup_{i=1}^k I_j$. For
$I=[a,b)\in\S$ and $\alpha\in \R$, we set $\alpha I:=[\alpha
a,\alpha b)$ and $I+\alpha:=[a+\alpha,b+\alpha)$. Similarly, for
$J\in\RR$, we define $\alpha J:=\alpha I_1\cup\cdots\cup \alpha I_k$ and
$J+\alpha:=(I_1+\alpha)\cup\cdots\cup (I_k+\alpha)$.

\begin{definition}\label{def:rare event process}
Given $J\in\RR$ and sequences $(u_n)_{n\in\N}$ and $(v_n)_{n\in\N}$, the \emph{rare event point process} (REPP) is defined by
counting the number of exceedances (or hits to $U(u_n)$) during the (re-scaled) time period $v_nJ\in\RR$. More precisely, for every $J\in\RR$, set
\begin{equation}
\label{eq:def-REPP} N_n(J):=\nn_{u_n}(v_nJ)=\sum_{j\in v_nJ\cap\N_0}\I_{X_j>u_n}.
\end{equation}
\end{definition}

Under dependence conditions similar to the ones previously analyzed, the REPP just defined converges in distribution to a standard Poisson process when no clustering is involved, and converges in distribution to a compound Poisson process with intensity $\vartheta$ and a geometric multiplicity distribution function otherwise. For the sake of completeness, we now define what we mean by a Poisson and a compound Poisson process. (See \cite{K86} for more details.)

\begin{definition}
\label{def:compound-poisson-process}
Let $Y_1, Y_2,\ldots$ be  an iid sequence of random variables with common exponential distribution of mean $1/\vartheta$. Let  $Z_1, Z_2, \ldots$ be another iid sequence of random variables, independent of the previous one, and with distribution function $\pi$. Given these sequences, for $J\in\RR$, set
$$
N(J)=\int \I_J\;d\left(\sum_{i=1}^\infty Z_i \delta_{Y_1+\ldots+Y_i}\right),
$$
where $\delta_t$ denotes the Dirac measure at $t>0$.  Whenever we are in this setting, we say that $N$ is a compound Poisson process of intensity $\vartheta$ and multiplicity d.f.\ $\pi$.
\end{definition}

\begin{remark}
\label{rem:poisson-process}
In this paper, the multiplicity will always be integer valued. This means that $\pi$ is completely defined by the values $\pi_k=\p(Z_1=k)$, for every $k\in\N_0$. Note that, if $\pi_1=1$ and $\vartheta=1$, then $N$ is the standard Poisson process and, for every $t>0$, the random variable $N([0,t))$ has a Poisson distribution of mean $t$.
\end{remark}

\begin{remark}
\label{rem:compound-poisson}
When clustering is involved, typically $\pi$ is a geometric distribution of parameter $\vartheta \in (0,1]$,  \ie $\pi_k=\vartheta(1-\vartheta)^{k-1}$, for every $k\in\N_0$. This means that, as in \cite{HV09}, the random variable $N([0,t))$ follows a P\'olya-Aeppli distribution
$$
\p(N([0,t)))=k=\e^{-\vartheta t}\sum_{j=1}^k \vartheta^j(1-\vartheta)^{k-j}\frac{(\vartheta t)^j}{j!}\binom{k-1}{j-1},
$$
for all $k\in\N$ and $\p(N([0,t)))=0=\e^{-\vartheta t}$.
\end{remark}

We are now in condition to state the dichotomy in terms of the convergence of REPP.

\begin{theorem}
\label{thm:dichotomy3}
Let  $T:\mathbb T^2\to \mathbb T^2$ be an Anosov linear diffeomorphism. Fix $\zeta\in \mathbb T^2$ and define the stochastic process $X_0, X_1, X_2, \ldots$ as in \eqref{eq:SP-def}. Let  $(u_n)_{n\in\N}$ be a sequence satisfying \eqref{eq:un} for some $\tau\geq 0$ and $(v_n)_{n\in\N}$ be given by $v_n=1/\p(X_0>u_n)$. Consider the REPP $N_n$ as in Definition~\ref{def:compound-poisson-process}. Then
\begin{enumerate}

\item if $\zeta$ is not periodic, the REPP $N_n$ converges in distribution to the standard Poisson process.

\item if $\zeta$ is periodic of prime period $q$, the REPP $N_n$ converges in distribution to a compound Poisson process with intensity $\vartheta$ given by \eqref{eq:EI-computed} and multiplicity d.f. $\pi$ given by \eqref{eq:multiplicity-dist}.
\end{enumerate}

\end{theorem}

As in the computation of the EI, the metric used has an important impact on the multiplicity distribution obtained when we have convergence of the REPP to a compound Poisson process. Indeed, a similar result for the adapted metric $d^*$ given in \eqref{eq:adjusted-metric} is as follows.

\begin{theorem}
\label{thm:dichotomy4}
Let  $T:\mathbb T^2\to \mathbb T^2$ be an Anosov linear diffeomorphism. Fix $\zeta\in \mathbb T^2$ and define the stochastic process $X_0, X_1, X_2, \ldots$ as in \eqref{eq:SP-def} for the observable $\varphi$ in \eqref{eq:observable} but evaluated with the metric $d^*$ instead. Let $(u_n)_{n\in\N}$ be a sequence satisfying \eqref{eq:un} for some $\tau\geq 0$ and $(v_n)_{n\in\N}$ be given by $v_n=1/\p(X_0>u_n)$. Consider the REPP $N_n$ as in Definition~\ref{def:rare event process}. Then
\begin{enumerate}

\item if $\zeta$ is not periodic, the REPP $N_n$ converges in distribution to the standard Poisson process.

\item if $\zeta$ is periodic of prime period $q$, the REPP $N_n$ converges in distribution to a compound Poisson process with intensity $\vartheta=1-|\lambda|^{-q}$ and geometric multiplicity d.f. $\pi^*$ given by $\pi^*(\kappa)=\vartheta (1-\vartheta)^{\kappa -1}$, for all $\kappa\in\N$.
\end{enumerate}

\end{theorem}

\subsection{Absence of clustering}

When condition $\D'_0(u_n)$ holds, there is no clustering and so we may benefit from a criterion, proposed by Kallenberg \cite[Theorem~4.7]{K86}, which applies only to simple point processes without multiple events. Accordingly, we can merely adjust condition $\D_0(u_n)$ to this scenario of multiple exceedances in order to prove that the REPP converges in distribution to a standard Poisson process. We denote this adapted condition by:

\begin{condition}[$D_3(u_n)$]\label{cond:D^*} Let $A\in\RR$ and $t\in\N$.
We say that $D_3(u_n)$ holds for the sequence
$X_0,X_1,\ldots$ if
\[ \left|\p\left(\{X_0>u_n\}\cap
  \{\nn_{u_n}(A+t)=0\}\right)-\p(\{X_0>u_n\})
  \p(\nn_{u_n}(A)=0)\right|\leq \gamma(n,t),
\]
where $\gamma(n,t)$ is nonincreasing in $t$ for each $n$ and
$n\gamma(n,t_n)\to0$ as $n\rightarrow\infty$ for some sequence
$t_n=o(n)$. (The last equality means that $t_n/n\to0$ as $n\to \infty$).
\end{condition}

In \cite[Theorem~5]{FFT10} it was proved a strengthening of~\cite[Theorem~1]{FF08a} which essentially says that, under conditions $D_3(u_n)$ and $\D'_0(u_n)$, the REPP  $N_n$ defined in \eqref{eq:def-REPP} converges in distribution to a standard Poisson process.

The proof of condition $D_3(u_n)$ follows after minor adjustments to the proof of $\D_q(u_n)$ in Section~\ref{sec:D}. Since condition $\D'_0(u_n)$ holds at every non-periodic point $\zeta$ (see Section~\ref{subsec:D'}), then for all such points $\zeta$  the corresponding REPP $N_n$ converges in distribution to a standard Poisson process.

\subsection{Presence of clustering}
\label{subsec:periodicity}

Condition $\D'_0(u_n)$ prevents the existence of clusters of exceedances, which implies that the EVL is  standard exponential $\bar H(\tau)=\e^{-\tau}$. When $\D'_0(u_n)$ fails, the clustering of exceedances is responsible for the appearance of a parameter $0<\vartheta<1$ in the EVL, called EI, and implies that, in this case, $\bar H(\tau)=\e^{-\vartheta \tau}$. In \cite{FFT12},  the authors established a connection between the existence of an EI less than 1 and a periodic behavior. This was later generalized for REPP in \cite{FFT13}.

For the convergence of the REPP when there is clustering, one cannot use the aforementioned criterion of Kallenberg  because the point processes are not simple anymore and possess multiple events. This means that a much deeper analysis must be done in order to obtain convergence of the REPP. This was carried out in \cite{FFT13} and we describe below the main results and conditions needed.

Let $\zeta$  be a periodic point of prime period $q$. Firstly, we consider the sequence $\left(U^{(\kappa)}(u_n)\right)_{\kappa\geq0}$ of nested balls centred at $\zeta$ given by
\begin{equation}
\label{eq:Uk-definition}
U^{(0)}(u_n)=U(u_n)=U_n
 \quad\text{and}\quad U^{(\kappa)}(u_n)=T^{-q}(U^{(\kappa-1)}(u_n))\cap U(u_n), \quad\text{for all $\kappa\in\N$.}
 \end{equation}
Then, for $i,\kappa,\ell,s\in\N\cup\{0\}$, we define the following events:
\begin{equation}\label{eq:Q-definition}
Q_{q,i}^\kappa(u_n):=T^{-i}\left(U^{(\kappa)}(u_n)-U^{(\kappa+1)}(u_n)\right).
\end{equation}

Note that $Q_{q,0}^0(u_n)=A_n^{(q)}$.
Besides,
$U_n=\bigcup_{\kappa=0}^\infty Q_{q,0}^\kappa(u_n),$
which means that the ball centred at $\zeta$  which corresponds to $U_n$ can be decomposed into a sequence of disjoint strips where $Q_{q,0}^0(u_n)$ are the most outward strips and the inner strips $Q_{q,0}^{\kappa+1}(u_n)$ are sent outward by $T^p$ onto the strips $Q_{q,0}^\kappa(u_n)$, i.e.,
$T^{q}(Q_{q,0}^{\kappa+1}(u_n))=Q_{q,0}^\kappa(u_n).$

We are now ready to state the adapted condition:

\begin{condition}[$D_q(u_n)^*$]\label{cond:Dp*}We say that $D_q(u_n)^*$ holds
for the sequence $X_0,X_1,X_2,\ldots$ if for any integers $t, \kappa_1,\ldots,\kappa_\zeta$, $n$ and
 any $J=\cup_{i=2}^q I_j\in \mathcal R$ with $\inf\{x:x\in J\}\ge t$,
 \[ \left|\p\left(Q_{q,0}^{\kappa_1}(u_n)\cap \left(\cap_{j=2}^q \nn_{u_n}(I_j)=\kappa_j \right) \right)-\p\left(Q_{q,0}^{\kappa_1}(u_n)\right)
  \p\left(\cap_{j=2}^q \nn_{u_n}(I_j)=\kappa_j \right)\right|\leq \gamma(n,t),
\]
where for each $n$ we have that $\gamma(n,t)$ is nonincreasing in $t$  and
$n\gamma(n,t_n)\to0$  as $n\rightarrow\infty$, for some sequence
$t_n=o(n)$.
\end{condition}

In \cite{FFT13}, for technical reasons only, the authors introduced a slight modification to $\D'_q(u_n)$. This condition was denoted by $D'_q(u_n)^*$ and it requires that
$$\lim_{n\rightarrow\infty}\,n\sum_{j=1}^{[n/k_n]}\p( Q_{q,0}(u_n)\cap
\{X_j>u_n\})=0,$$ which holds whenever condition $\D'_q(u_n)$ does.

From the study developed in \cite{FFT13}, it follows that if $X_0, X_1, \ldots$ satisfies conditions $D_q(u_n)^*$, $D'_q(u_n)^*$ and $\lim_{n\to\infty}\sum_{\kappa\geq1}\p(U^{(\kappa)}(u_n))=0$, where $(u_n)_{n\in\N}$ is such that \eqref{eq:un} holds, then the REPP $N_n$ converges in distribution to a compound Poisson process with intensity $\vartheta$ given by \eqref{eq:OBrien-EI} and multiplicity d.f. $\pi$ given by:
\begin{equation}
\label{eq:multiplicity}
\pi(\kappa)=\lim_{n\to\infty}\frac{
\left(\p(Q_{q,0}^{\kappa-1}(u_n))-\p(Q_{q,0}^\kappa(u_n))\right)}{
\p(Q_{q,0}^0(u_n))}.
\end{equation}

\subsection{Computing the multiplicity distribution}

\begin{figure}
\includegraphics[width=\linewidth]{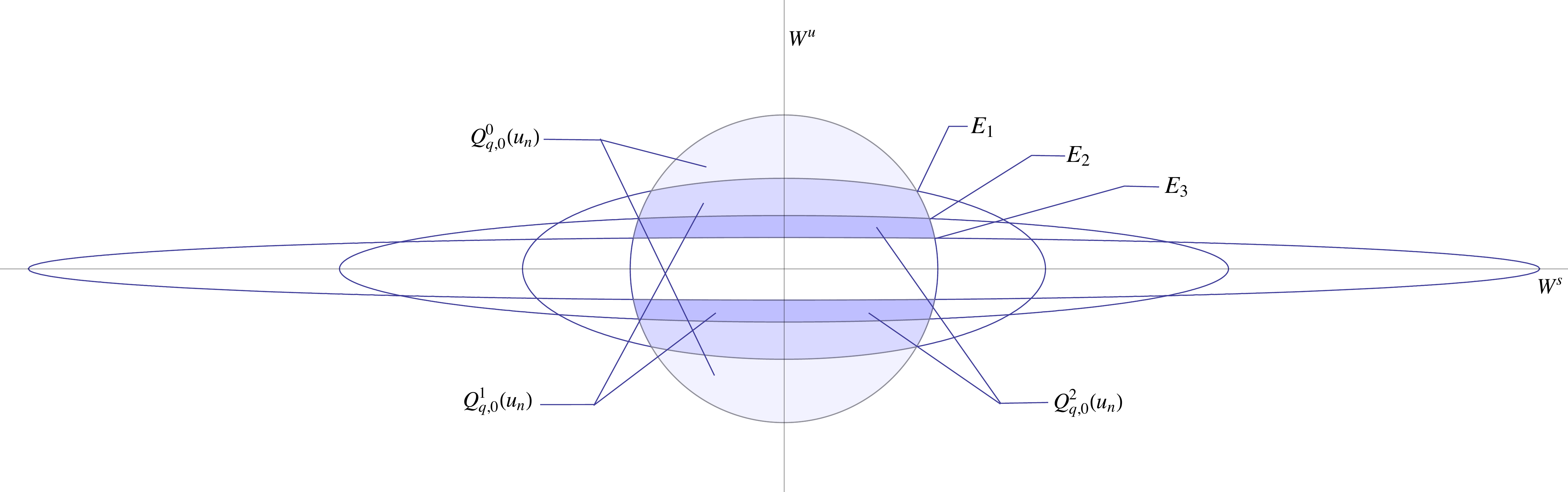}
\caption{The sets $Q_{q,0}^\kappa(u_n)$ for the Euclidean metric.}
\label{fig:Q}
\end{figure}

In order to compute the multiplicity distribution, we need to calculate $m(Q_{q,0}^\kappa)$. To do so, let $E_i$ denote one of the four intersection points of $\partial T^{-i\,q}(U_n)$ with $\partial U_n$.  Let $(x_{E_i},y_{E_i})$ denote the coordinates of $E_i$ on the unstable and stable directions, which correspond to the vertical and horizontal directions, respectively, in Figure~\ref{fig:Q}. We assume that the choice of the $E_i$'s among the possible four points is made so that $x_{E_i},y_{E_i}>0$. Hence, 
$$
y_{E_i}=\frac{s_n|\lambda|^{iq}}{\sqrt{1+\lambda^{2iq}}},
$$
because $\partial T^{-i\,q}(U_n)$ is described by the equation $\lambda^{2iq}x^2+\lambda^{-2iq}y^2=s_n^2$. Therefore, for $\kappa\geq 1$, we may write
\begin{align*}
m(Q_{q,0}^\kappa(u_n))&=4\int_{0}^{y_{E_i}} \lambda^{-2iq}\sqrt{s_n^2-\lambda^{-2iq}y^2}dy+4\int_{y_{E_i}}^{y_{E_{i+1}}} \sqrt{s_n^2-y^2}dy\\&
\quad-4\int_{0}^{y_{E_{i+1}}} \lambda^{-2(i+1)q}\sqrt{s_n^2-\lambda^{-2(i+1)q}y^2}dy\\
&=2s_n^2\left(\arcsin\frac{|\lambda|^{(i+1)q}}{\sqrt{1+\lambda^{2(i+1)q}}}-\arcsin\frac{1}{\sqrt{1+\lambda^{2(i+1)q}}}\right)\\
&\quad +2s_n^2\left(\arcsin\frac{1}{\sqrt{1+\lambda^{2iq}}}-\arcsin\frac{|\lambda|^{iq}}{\sqrt{1+\lambda^{2iq}}}\right)
\end{align*}
Recall that $Q_{q,0}^0(u_n)=A_n^{(q)}$ and its measure has been computed in \eqref{eq:Aq}. Hence, using formula \eqref{eq:multiplicity}, we obtain
\begin{align}
\label{eq:multiplicity-dist}
\pi(\kappa)&=\frac{\left(\arcsin\frac{1}{\sqrt{1+\lambda^{2(\kappa-1)q}}}-\arcsin\frac{|\lambda|^{(\kappa-1)q}}{\sqrt{1+\lambda^{2(\kappa-1)q}}}\right)+\left(\arcsin\frac{1}{\sqrt{1+\lambda^{2(\kappa+1)q}}}-\arcsin\frac{|\lambda|^{(\kappa+1)q}}{\sqrt{1+\lambda^{2(\kappa+1)q}}}\right)}{\arcsin\,\frac{|\lambda|^q}{\sqrt{1+\lambda^{2q}}}-\arcsin\,\frac{1}{\sqrt{1+\lambda^{2q}}}}\nonumber\\
&\quad +\frac{2\left(\arcsin\frac{|\lambda|^{\kappa q}}{\sqrt{1+\lambda^{2\kappa q}}}-\arcsin\frac{1}{\sqrt{1+\lambda^{2\kappa q}}}\right)}{\arcsin\,\frac{|\lambda|^q}{\sqrt{1+\lambda^{2q}}}-\arcsin\,\frac{1}{\sqrt{1+\lambda^{2q}}}}
\end{align}

\begin{figure}
\includegraphics[width=\linewidth]{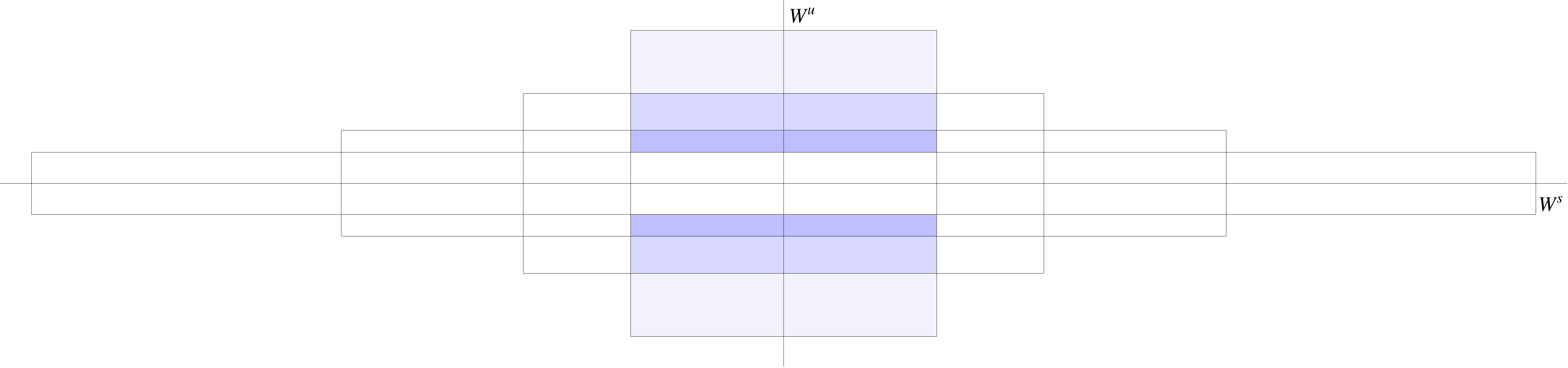}
\caption{The sets $Q_{q,0}^\kappa(u_n)$ in the adapted metric.}
\label{fig:Q-metric2}
\end{figure}

\begin{remark}
\label{rem:pi-metric}
As for the EI (see Remark~\ref{rem:EI-metric}), the multiplicity distribution depends on the metric used. If we were to use the adapted metric $d^*$, in which balls correspond to squares whose sides are lined up with the stable and unstable directions, then we would obtain a geometric distribution for the multiplicity, \ie $\pi^*(\kappa)=\vartheta^*(1-\vartheta^*)^{\kappa -1}$, where $\vartheta^*=1-|\lambda|^{-q}$.
\end{remark}

Observe also that in the Euclidean metric case we do not obtain a geometric distribution, as one did for repelling periodic points (see \cite{FFT13}) or in the case of the adapted metric mentioned in Remark~\ref{rem:pi-metric}. However, taking into account that
$$
\arcsin\left(\frac{x}{\sqrt{1+x^2}}\right)= \frac\pi 2-\frac1x+o(1/x)\quad\mbox{and}\quad \arcsin\left(\frac{1}{\sqrt{1+x^2}}\right)= \frac1x+o(1/x),
$$
one deduces that
$$
\lim_{\kappa\to\infty}\frac{\pi(\kappa+1)}{\pi(\kappa)}=|\lambda|^{-q}=\frac{\pi^*(\kappa+1)}{\pi^*(\kappa)},
$$
where $\pi^*(\kappa)$ is as given in Theorem~\ref{thm:dichotomy4}. Hence, we obtained a multiplicity distribution different from the geometric distribution but which is not that far from it.

The proof of condition $D_q(u_n)^*$ follows with some adjustments from the proof of $\D_q(u_n)$ in Section~\ref{sec:D}. Since condition $\D'_q(u_n)$ holds at every periodic point $\zeta$, as was shown in Section~\ref{subsec:D'}, then condition $D'_q(U_n)^*$ from \cite{FFT13} also holds for all such points $\zeta$. Let $C_n^\kappa$ denote a rectangle, centered at $\zeta$, with its sides lined up with the stable and unstable direction, of length $s_n$ on the stable direction and $|\lambda|^{-\kappa q} s_n$ on the unstable direction. Observe that $U^{(\kappa)}(u_n)\subset C_n^{\kappa}$, which implies that 
$$m(U^{(\kappa)}(u_n))\leq |\lambda|^{-\kappa q} s_n^2$$
and consequently 
$$\lim_{n\to\infty}\sum_{\kappa\geq 1} m(U^{(\kappa)}(u_n)=\lim_{n\to\infty} \frac{1}{1-|\lambda|^{-q}}s_n^2=0.$$
So, at every periodic point $\zeta$  of prime period $q$, the respective REPP $N_n$ converges in distribution to a compound Poisson process with intensity $\vartheta$ given by \eqref{eq:EI-computed} and multiplicity d.f. $\pi$ given by \eqref{eq:multiplicity-dist}.

\bibliographystyle{abbrv}

\bibliography{cat-map}

\end{document}